\documentclass[11pt, twoside, a4paper]{article}
\usepackage[utf8]{inputenc}
\usepackage{amsmath, amssymb, amsthm}            
\usepackage{amstext, amsfonts, a4}
\usepackage[english]{babel}
\usepackage{xcolor}
\usepackage{bbm}
\usepackage{combelow}
\usepackage{tikz}

\usepackage[margin=2.5cm]{geometry}
\usepackage{dsfont}

\usepackage{xcolor}
\usepackage[hidelinks]{hyperref}
\hypersetup{
 colorlinks,
 linkcolor={red!50!black},
 citecolor={blue!50!black},
 urlcolor={blue!80!black}
}

\usepackage[mathlines, pagewise]{lineno}
\numberwithin{equation}{section}

\renewcommand{\Re}{\operatorname{Re}}
\renewcommand{\Im}{\operatorname{Im}}
\newcommand{\ic}{\mathbf{i}}

\newcommand{\tend}[2]{\underset{#1\to #2}{\longrightarrow} }

\newcommand{\ntriple}[1]{{\left\vert\kern-0.25ex\left\vert\kern-0.25ex\left\vert #1 
    \right\vert\kern-0.25ex\right\vert\kern-0.25ex\right\vert}}

\newcommand{\der}[2]{\frac{\dd #1}{\dd #2}}

\newcommand{\dd}{\mathrm{d}}
\newcommand{\Diff}{\mathrm{D}}

\newcommand{\supp}{\mathrm{supp}\hspace{1mm}}

\newcommand{\eps}{\varepsilon}

\newcommand{\R}{\mathbb{R}}
\newcommand{\C}{\mathbb{C}}

\DeclareMathOperator{\vect}{span}

\renewcommand{\tilde}{\widetilde}

\theoremstyle{plain}
	\newtheorem{theorem}{Theorem}[section]
	\newtheorem{proposition}[theorem]{Proposition}    
	\newtheorem{lemma}[theorem]{Lemma}

	\newtheorem{remark}[theorem]{Remark}
	
    \newtheorem{hypothesis}[theorem]{Hypothesis}
	
\theoremstyle{definition}
	\newtheorem{definition}[theorem]{Definition}

\title{Long-time Confinement near Special Vortex Crystals}
\author{Martin Donati\footnote{Université Grenoble Alpes, Institut Fourier, 100 rue des Mathématiques, 38610 Gières, France.}}
\date{\today}

\begin{document}

\maketitle

\begin{abstract}
    In this paper, we control the growth of the support of particular solutions to the Euler two-dimensional equations, whose vorticity is concentrated near special vortex crystals. These vortex crystals belong to the classical family of regular polygons with a central vortex, where we choose a particular intensity for the central vortex to have strong stability properties. A special case is the regular pentagon with no central vortex which also satisfies the stability properties required for the long-time confinement to work.
\end{abstract}

\section{Introduction}

In this paper, we are interested in the long-time stability of vortex structures in two-dimensional incompressible and inviscid fluids. It is long observed that from turbulent flows can emerge naturally large stable vortex structures, see for instance \cite{Driscoll_et_al_1995_Relaxation,Schecter_Driscol_et_al_1999_Vortex,Siegelman_Young_Ingersoll_2022_Polar_vortex_crystals,Drivas_Elgindi_2023_Singularity}. These structures can live for a particularly long time within the fluid. If the initial data is well chosen, this time can even be infinite if the viscosity vanishes, see for instance \cite{Smets_VanSchaftingen_2010}, but also very long when the viscosity is only small, see \cite{Dolce_Gallay_2024_Long_way}.

To study these structures, we consider a fluid with a finite number of small regions where its vorticity is very high, and vanishes elsewhere, so that the fluid mostly consists in a few vortices. We look at the asymptotic where the size of those regions goes to 0, but the \emph{intensity} of the vortices, which is the integral of the vorticity over each region, is fixed. At the limit where the vortices are infinitely concentrated, their motion is given by the \emph{point-vortex dynamics}, see for instance \cite{Marchioro_Pulvirenti_1993}. In that case, the whole evolution of the fluid is reduced to this system of ODE.

This dynamical system has been extensively studied, see for instance \cite{Aref_2007} and references therein. Numerous particular solutions are known. We are interested in a particular class of solutions which are called \emph{vortex crystals}. These configurations of vortices move rigidly: distances between points and angles are conserved through time. This is exactly the kind of configuration arising naturally from turbulent flows through the effect of merging smaller vortices. When accidentally "trapped" into a vortex crystal, vortices cannot get close enough to other vortices to keep merging, letting this vortex crystal survive until the effect of viscosity spreads them enough to restart the merging process.

Vortex crystals have also been studied, see \cite{AREF_2003_Vortex_Crystals}, and in particular the question of stability of these configurations, to explain why they could live so long even within the remains of a turbulent background. In \cite{Cabral_Schmidt_1999_Stability_N+1_vortex}, the stability of particular configurations is proved. These configurations consist of $N$ vortices, $N-1$ of which are placed at the vertices of a regular polygon and have the same intensity, and the last one is placed at the center of it with an intensity of his own. When the central vorticity lies within a certain range depending on the number of points and the intensity of the exterior vortices, then \cite{Cabral_Schmidt_1999_Stability_N+1_vortex} proves the linear and nonlinear stability of the configurations.

The desingularization of such configurations, namely proving the existence of solutions to the Euler equations concentrated near these configurations of point-vortices has been long done, see \cite{Turkington_1985_Corotating}. In the present paper, we try to push further the analysis and prove that any solution to the Euler equation concentrated near this type of configuration remains concentrated for a very long time following the vortex crystal prediction. More precisely, if the support of the initial vorticity lies within disks of size $\eps$, we prove that it remains within disks of size $\eps^\beta$, with $\beta<1/2$, for a time at least of order $\eps^{-\alpha}$ for some $\alpha >0$. This extends previous results, see \cite{marchiorobutta2018,Donati_Iftimie_2020}, where such long time confinement was obtained for configurations satisfying either $N=1$ (only one vortex), or that point-vortices had to move away from each other very fast. Here, we prove that we need neither of those assumptions, but instead that the vortices are disposed according to the previously described polygonal configuration.

It is important to mention that unlike the stability result proved in \cite{Cabral_Schmidt_1999_Stability_N+1_vortex}, we do not prove our result for a whole range of intensity of the central vortex, but only for one precise intensity, which gives in some sense the best possible stability. This choice is made so that the strain that each vortex undergoes due to the effect of the others vanishes at first order. These configurations are depicted in Figure~\ref{fig:config}. It requires at least 4 vortices, and it is worth mentioning that in the special case of vortices disposed of as a regular pentagon, one has to choose the intensity of the central vortex to be exactly 0 so that our method works. This means that the pentagon, with no central vortex, is very stable for the confinement problem. For every other number of exterior vortices, we obtain a unique choice of non-vanishing intensity for the central vortex, that is negative for $N \in \{4,5\}$ and positive for $N \ge 7$.

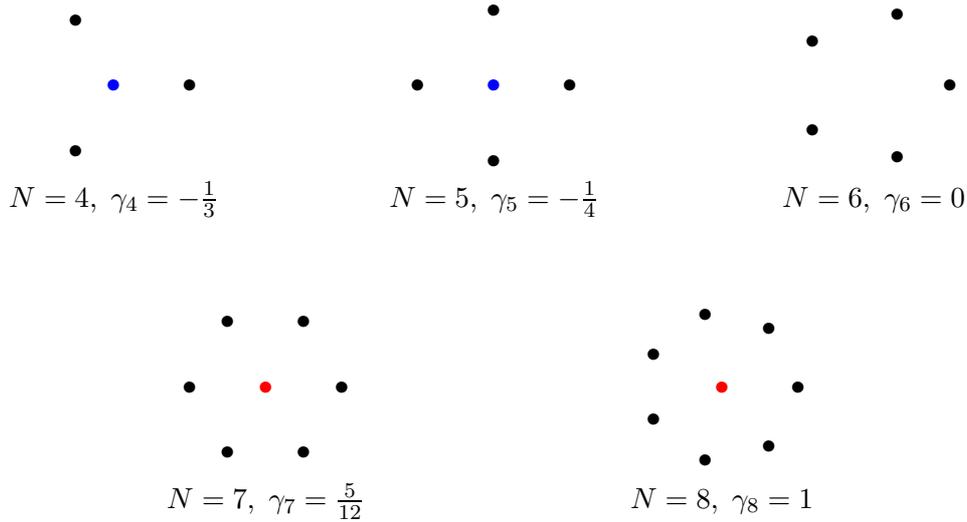
\begin{figure}
    \centering
     \begin{tikzpicture}

        \draw (0,0) node {$\textcolor{blue}{\bullet}$};
        \draw (0,-1.5) node {$N=4, \; \gamma_4 = -\frac{1}{3}$};
        \draw (0+1,0) node {$\bullet$};
        \draw (0-0.5,0.866) node {$\bullet$};
        \draw (0-0.5,-0.866) node {$\bullet$};

        \draw (5,0) node {$\textcolor{blue}{\bullet}$};
        \draw (5,-1.5) node {$N=5, \; \gamma_5 = -\frac{1}{4}$};
        \draw (5+1,0) node {$\bullet$};
        \draw (5-1,0) node {$\bullet$};
        \draw (5,1) node {$\bullet$};
        \draw (5,-1) node {$\bullet$};

        \draw (10+0,-1.5-0) node {$N=6, \; \gamma_6 = 0$};
        \draw (10+1,0-0) node {$\bullet$};
        \draw (10+0.309,0.951-0) node {$\bullet$};
        \draw (10-0.808,0.589-0) node {$\bullet$};
        \draw (10-0.808,-0.589-0) node {$\bullet$};
        \draw (10+0.309,-0.951-0) node {$\bullet$};

        \draw (2+0,0-4) node {$\textcolor{red}{\bullet}$};
        \draw (2,-1.5-4) node {$N=7, \; \gamma_7 = \frac{5}{12}$};
        \draw (2+1,0-4) node {$\bullet$};
        \draw (2+0.5,0.866-4) node {$\bullet$};
        \draw (2-0.5,0.866-4) node {$\bullet$};
        \draw (2-1,0-4) node {$\bullet$};
        \draw (2-0.5,-0.866-4) node {$\bullet$};
        \draw (2+0.5,-0.866-4) node {$\bullet$};

        \draw (8+0,0-4) node {$\textcolor{red}{\bullet}$};
        \draw (8,-1.5-4) node {$N=8, \; \gamma_8 = 1$};
        \draw (8+1,0-4) node {$\bullet$};
        \draw (8+0.62,0.78-4) node {$\bullet$};
        \draw (8-0.22,0.97-4) node {$\bullet$};
        \draw (8-0.9,0.43-4) node {$\bullet$};
        \draw (8-0.9,-0.43-4) node {$\bullet$};
        \draw (8-0.22,-0.97-4) node {$\bullet$};
        \draw (8+0.62,-0.78-4) node {$\bullet$};
    \end{tikzpicture}
    \caption{Various polygonal vortex crystals of radius 1 and exterior intensities 1, where $\gamma_N$ designates the intensity of the central vortex.}
    \label{fig:config}
\end{figure}

\section{Modelling and main result}

We consider the two-dimensional incompressible Euler equations in vorticity form
\begin{equation}\tag{Eu}\label{Eu}
    \begin{cases}
        \partial_t \omega(x,t) + u(x,t) \cdot \nabla \omega(x,t) = 0 & \text{ on } \R^2 \times \R_+ \vspace{2mm} \\
        \displaystyle u(x,t) = \int_{\R^2} K(x-y) \omega(y,t)\dd y = \big(K\star \omega(t)\big)(x) & \text{ on } \R^2 \times \R_+,
    \end{cases}
\end{equation}
where
\begin{equation*}
    \forall x \in \R^2 \setminus \{0\}, \quad K(x) = \frac{x^\perp}{2\pi |x|^{2}}, \qquad x^\perp := (-x_2,x_1).
\end{equation*}

For $\omega_0 \in L^1 \cap L^\infty(\R^2)$, it is now well known since \cite{YUDOVICH19631407} that these equations have a unique global solution $\omega \in L^\infty\big(\R_+, L^1 \cap L^\infty(\R^2)\big)$, see for instance \cite[Chapter 2]{Marchioro_Pulvirenti_1993}. In this paper, we are interested in the long-time behavior of particular families of solutions. More precisely, for a given choice of $N$ different points $z_1,\ldots,z_N \in \R^2$ and intensities $\gamma_1,\ldots,\gamma_N \in \R\setminus\{0\}$, we consider a family of initial data $(\omega_0^\eps)_{\eps >0}$ satisfying the following hypothesis.
\begin{hypothesis}\label{hyp:omega}
    There exist constants $M$ and $\eta \ge 2$ such that for every $\eps > 0$,
    \begin{itemize}
        \item $\omega_0^\eps \in L^1 \cap L^\infty(\R^2)$,
        \item $\displaystyle \omega_0^\eps = \sum_{i=1}^N \omega_{0,i}^\eps \quad $ where $\displaystyle \quad \frac{\omega_{0,i}^\eps}{\gamma_i} \ge 0$ and $\displaystyle\int_{\R^2} \omega_{0,i}^\eps(x)\dd x = \gamma_i$,
        \item $\supp \omega_{0,i} \subset D(z_i,\eps)$,
        \item $ |\omega_0^\eps | \le M \eps^{-\eta}$,
    \end{itemize}
    where we denote by $D(z,r) := \{ x \in \R^2 \, , \, |x-z| < r\}$.
\end{hypothesis}
In particular, such a family $(\omega_0^\eps)_{\eps > 0}$ satisfies that $$\omega_0^\eps \tend{\eps}{0} \sum_{i=1}^N \gamma_i \delta_{z_i}$$ weakly in the sense of measures. It is now well known, see \cite{marchioro1993VorticiesAndLocalization}, that in that case, the solution $\omega^\eps$ of the problem
\begin{equation}\label{eq:du_pb}
    \begin{cases}
        \partial_t \omega^\eps + u^\eps \cdot \nabla \omega^\eps = 0 \vspace{2mm} \\
        u^\eps(t) = K \star \omega^\eps(t), \vspace{2mm} \\
        \omega^\eps(0) = \omega_0^\eps
    \end{cases}
\end{equation}
satisfies on any time interval $[0,T]$ that 
\begin{equation}\label{eq:CV_omega}
    \omega^\eps(t) \tend{\eps}{0} \sum_{i=1}^n \gamma_i \delta_{z_i(t)}
\end{equation}
weakly in the sense of measures, where the $t\mapsto z_i(t)$ are given by the following system:
\begin{equation}\label{PVS}\tag{PVS}
    \der{}{t} z_i(t) = \sum_{j \neq i} \gamma_j K\big(z_i(t)-z_j(t)\big) = \sum_{j \neq i} \frac{\gamma_j}{2\pi} \frac{\big(z_i(t)-z_j(t)\big)^\perp}{\big| z_i(t)-z_j(t)\big|^2}.
\end{equation}
This dynamics is called the \emph{point-vortex system}, as it corresponds to the ideal dynamics of vortices in a fluid whose vorticity is infinitely concentrated in some points.

Solutions of this dynamics are not always global in time as \emph{collapses} of point-vortices can happen. Indeed, if there exist two indices $i \neq j$ and a time $T< \infty$ such that
\begin{equation*}
    \liminf_{t\to T} |z_i(t)- z_j(t)| = 0,
\end{equation*}
then the equations~\eqref{PVS} blow-up at time $T$. However, collapses are known to be exceptional, see \cite{Marchioro_Pulvirenti_1993}, under the standard assumption on the intensities that
\begin{equation*}
    \forall P \subset \{1,\ldots,N\}, \quad \sum_{i \in P} \gamma_i \neq 0.
\end{equation*}
In the following, we will only deal with configurations that do not give rise to such collapses, so that the solution of the point-vortex dynamics will be global in time. For additional information on the point-vortex dynamics, we refer the reader to \cite{Aref_2007}.

Recent works have been more precise on the convergence \eqref{eq:CV_omega}. In particular, since the initial datum is compactly supported, the solution remains compactly supported at all times. Therefore, one can define for any $\beta < 1$ the time
\begin{equation*}
    \tau_{\eps,\beta} = \sup \left\{ t \ge 0 \, , \, \forall s \in [0,t] \, , \, \supp \omega^\eps(s) \subset \bigcup_{i=1}^N D(z_i(s),\eps^\beta) \right\},
\end{equation*}
which describes how long the solution remains supported within disks of size $\eps^\beta$ around the point-vortex solution. In \cite{marchiorobutta2018}, the authors prove that in general, this time of confinement is at least of order $|\ln \eps|$.

\begin{theorem}[\cite{marchiorobutta2018}]
    Let $(\omega_0^\eps)_{\eps>0}$ satisfying Hypothesis~\ref{hyp:omega} for some $z_1,\ldots,z_N$ and $\gamma_1,\ldots,\gamma_N$ giving rise to a global solution of the point-vortex dynamics~\eqref{PVS} such that
    \begin{equation*}
        \inf_{t\ge 0} \min_{i\neq j} |z_i(t)-z_j(t)| >0.
    \end{equation*} Then for any $\beta<1/2$ there exists a constant $C$ such that for any $\eps > 0$ small enough, the solution $\omega^\eps$ of the problem~\eqref{eq:du_pb} satisfies $\tau_{\eps,\beta} \ge C|\ln\eps|$.
\end{theorem}

This logarithmic bound is optimal in general, see \cite{Donati_2024_DCDS}. However, in the same article, Buttà and Marchioro observed that for some special configurations of point-vortices, this bound can be improved.

\begin{theorem}[\cite{marchiorobutta2018}] \label{theo:BM_2018_strong}
    Assume that $z_1,\ldots,z_N$ and $\gamma_1,\ldots,\gamma_N$ are chosen such that the solution of the system~\eqref{PVS} is a self-similar expanding configuration. Then for any $\beta<1/2$ there exists a constant $\alpha > 0$ such that for any $\eps > 0$ small enough, the solution $\omega^\eps$ of the problem~\eqref{eq:du_pb} satisfies $\tau_{\eps,\beta} \ge \eps^{-\alpha}$.
\end{theorem}
The fact that the configuration is self-similar is not necessary for the proof of Theorem~\ref{theo:BM_2018_strong} but instead, what is required is that for every $i \neq j$, and every $t \ge 0$,
\begin{equation}\label{eq:config_infty}
    |z_i(t) - z_j(t)| \ge C_{i,j} \sqrt{t+1} \tend{t}{+\infty} + \infty. 
\end{equation}
There are other configurations found in \cite{marchiorobutta2018} that satisfy this improved bound on the time of confinement when one only considers a single blob of vorticity $(N=1)$, in the plane, or at the center of a disk. In \cite{Donati_Iftimie_2020}, it is proved that, for $N=1$ in some special bounded and simply connected domains with a suitable choice of $z_1$, the solution satisfies that $\tau_{\eps,\beta} \ge \eps^{-\alpha}$ for any $\alpha < \min(\beta,2-4\beta)$.

\medskip

In this paper, we prove that the same improved bound on the confinement time can be obtained near configurations of point-vortices that do not need to satisfy relation~\eqref{eq:config_infty} nor $N=1$, but instead satisfy some strong stability properties. We consider the following configurations.

Let $\gamma_1 = \ldots = \gamma_{N-1} = 1$ and $\gamma_N \in \R$ to be chosen later. We identify $\C = \R^2$ by considering that $ a + \ic b = \begin{pmatrix} a \\ b \end{pmatrix}$. Let $\zeta = e^{\ic \frac{2\pi}{N-1}}$. We set
\begin{equation}\label{eq:config}
    \begin{cases}
        z_k^* = \zeta^k, \quad \forall k \in \{1,\ldots, N-1\},\\
        z_N^* = 0.
    \end{cases}
\end{equation}
This configuration is a relative equilibrium of the point-vortex dynamics in the following sense.
\begin{definition}\label{def:rel_eq}
    We say that a point $(z_1^*,\ldots,z_N^*) \in (\R^2)^N$ is a relative equilibrium of the point-vortex dynamics if the solution $t \mapsto (z_1(t),\ldots,z_N(t))$ of~\eqref{PVS} starting from $(z_1^*,\ldots,z_N^*)$ satisfies for every $i \in \{1,\ldots,N\}$ and every $t \in \R$ that $z_i(t) = e^{\ic \nu t} z_i^* $ for some $\nu$ that depends neither on $i$ nor on $t$.
\end{definition}
The solution arising from a relative equilibrium is called a \emph{vortex crystal}.
The stability of the configuration given in~\eqref{eq:config}, depending on $\gamma_N$, has been studied in particular in \cite{Cabral_Schmidt_1999_Stability_N+1_vortex} where the following result\footnote{In \cite{Cabral_Schmidt_1999_Stability_N+1_vortex}, the number of vortices is denoted $N+1$, hence the difference in the formulation of \cite{Cabral_Schmidt_1999_Stability_N+1_vortex} with the present statement of their result.} was shown.
\begin{theorem}[{\cite[Theorem 5.1]{Cabral_Schmidt_1999_Stability_N+1_vortex}}]\label{theo:CS}
    For any $N \ge 4$, if
    \begin{equation}\label{eq:range_gammaN}
        \begin{split}
            (N^2 - 10N + 17)/16 < \gamma_N < (N-2)^2/4 & \quad \text{ when $N$ is odd} \\
            (N^2 - 10N + 16)/16 < \gamma_N < (N-2)^2/4 & \quad \text{ when $N$ is even},
        \end{split}
    \end{equation}
    then the configuration~\eqref{eq:config} is linearly and nonlinearly stable.
\end{theorem}
The precise meaning of stability in this result will be given in Section~\ref{sec:results}.
The main result of this paper is the following.
\begin{theorem}\label{theo:main}
    For every $N \ge 4$, there exists $\gamma_N \in \R$ satisfying relations~\eqref{eq:range_gammaN} such that for any $(\omega_0^\eps)_{\eps>0}$ satisfying Hypothesis~\ref{hyp:omega} for the initial configuration~\eqref{eq:config}, for any $\beta<1/2$, for every $\alpha<\min(\beta/2,2-4\beta)$ and for any $\eps > 0$ small enough, the solution $\omega^\eps$ of the problem~\eqref{eq:du_pb} satisfies $\tau_{\eps,\beta} \ge \eps^{-\alpha}$.
\end{theorem}
This statement must be associated to the following remark.
\begin{remark}\label{rem:main}
    Our method provides for every $N \ge 4$ a unique value of $\gamma_N$ such that the bound of the confinement holds. Moreover, this value is 0 if and only if $N=6$: in that case, Theorem~\ref{theo:main} should be understood as a confinement result around the configuration consisting of a regular pentagon of vortices of equal intensity (with no central vortex).
\end{remark}
We recall that the configurations for the first values of $N$, with the associated $\gamma_N$, are depicted in Figure~\ref{fig:config}. The explicit formula for $\gamma_N$ is given at relation~\eqref{def:gammaN}.

\medskip

Theorem~\ref{theo:main} is the result of two separate results which are stated in detail in Section~\ref{sec:results}. The first one is a long-time confinement result around a general vortex crystal satisfying some precise hypotheses, and the second is the statement that the configurations we are considering satisfy those hypotheses. As stated in Remark~\ref{rem:main}, our methods provide a unique $\gamma_N$ such that the long-time confinement result holds, but we do not claim that the confinement result is false for other values of $\gamma_N$ within the range given by Theorem~\ref{theo:CS}. However, outside of that range, if the configuration is unstable then the instability prevents the long time confinement from holding, see \cite{Donati_2024_DCDS}.

\medskip

The plan of the paper is the following. In Section~\ref{sec:results} we introduce the appropriate definitions and hypotheses then state and discuss in detail the core results of the paper. In Section~\ref{sec:stab}, we establish an important lemma on the consequence of the stability hypothesis on the dynamics of the center of vorticity of each blob. In Section~\ref{sec:confinement} we prove the confinement result, and in Section~\ref{sec:construction}, we compute a value of $\gamma_N$ such that the configurations given at~\eqref{def:rel_eq} satisfy all the required hypotheses.

\section{Detailed hypotheses and results}\label{sec:results}

\subsection{Hypotheses on the configuration}

We start by stating the general assumptions required on the initial configuration of point vortices around which the confinement result will be obtained. The notations and definitions are inspired from~\cite{Roberts_2013_Stability_of_Relative_Eq} to which we refer the reader interested in a deeper description of the stability of relative equilibria of the point-vortex dynamics.
\medskip

Let $N \ge 1$ and $\gamma_1,\ldots,\gamma_N \in \R\setminus\{0\}$. For some definitions, all the intensities must be non-vanishing, so one simply needs to keep in mind that taking a vanishing intensity for one vortex means removing it so that all the intensities remain non-vanishing.

We now write the point-vortex dynamics as a single equation in $\R^{2N}$. For every $i \in \{1,\ldots,N\}$ we write that $z_i = (x_i,y_i) \in \R^2$ and denote by $Z$ the $2N$-dimensional vector $(x_1,y_1,\ldots,x_N,y_N)$, that we also denote by abuse of notation $Z = (z_1,\ldots,z_N)$. We equip $\R^{2N}$ with the following norm
\begin{equation*}
    |Z| := \max_{1\le i \le N} |z_i|.
\end{equation*}
We denote by $J_{K} : \R^2 \to \mathcal{M}_2(\R)$ the Jacobian matrix of $K$, the vector field defined below equations~\eqref{Eu}.
We can now formulate the first hypothesis.
\begin{hypothesis}\label{hyp:strong_stabilité}
    We assume that $t \mapsto Z^*(t)$ is a solution of the point-vortex equations~\eqref{PVS} such that
    \begin{equation*}
        \forall t \ge 0, \; \forall i \in\{1,\ldots, N\}, \quad \sum_{j\neq i} \gamma_j J_{K}(z_i^*(t)-z_j^*(t)) = 0. 
    \end{equation*}
\end{hypothesis}
We introduce some other notations. Recall that the Hamiltonian
\begin{equation*}
    H(Z) := \frac{1}{2}\sum_{j\neq i} \gamma_i\gamma_j \ln |z_i-z_j| 
\end{equation*}
is a constant of motion. With the $2N\times 2N$ symplectic matrix $J$ given by
\begin{equation*}
    J := \begin{pmatrix} 0 & -1 & & & \\ 1 & 0 & & (0) & \\ & & \ddots & & \\ & (0) & & 0 & -1 \\ & & & 1 & 0\end{pmatrix},
\end{equation*}
with $\Gamma := \mathrm{diag} (\gamma_1,\gamma_1,\ldots,\gamma_N,\gamma_N)$ and with the notation $\nabla = (\partial_{x_1},\partial_{y_1},\ldots,\partial_{x_N},\partial_{y_N})$, the point-vortex equations~\eqref{PVS} can be written as
\begin{equation*}
    \Gamma \der{}{t} Z(t) = J\nabla H (Z(t)).
\end{equation*}
In a rotating frame of constant velocity $\nu$, (namely by switching to $\tilde{Z} = e^{\nu tJ}Z$) this equation becomes
\begin{equation}\label{eq:PVS_rot}
    \Gamma \der{}{t} Z(t) = J(\nabla H (Z(t)) + \nu \Gamma Z(t)).
\end{equation}
Therefore, a configuration $Z^*$ is a relative equilibrium in the sense of Definition~\ref{def:rel_eq} if and only if it is a rest point of this last equation, hence satisfies that
\begin{equation*}
    \nabla H (Z^*) + \nu \Gamma Z^* = 0.
\end{equation*}
The angular velocity $\nu$ is then necessarily given by
\begin{equation}\label{def:nu}
    \nu = - \frac{\sum_{i\neq j} \gamma_i \gamma_j}{2\sum_{i=1}^n \gamma_i |z_i^*|^2}.
\end{equation}
Equation~\eqref{eq:PVS_rot} written near an equilibrium $Z^*$ reads
\begin{equation*}
    \der{}{t} \big(Z(t)-Z^*) = \Gamma^{-1} J \Diff^2 H(Z^*) (Z(t)-Z^*)+ \nu J (Z(t)-Z^*) + \mathcal{O}(|Z(t)-Z^*|^2).
\end{equation*}
We introduce the stability matrix
\begin{equation*}
    S = \Gamma^{-1} J \Diff^2 H(Z^*) + \nu J.
\end{equation*}
For some relative equilibrium $Z^*$ let
\begin{equation*}
    V = \vect \{ Z^*, JZ^* \}
\end{equation*}
and let
\begin{equation*}
    V^\perp := \{ Y \in \R^{2N} \, , \, { }^t Y \Gamma Z = 0 \, , \, \forall Z \in V \}.
\end{equation*}
\begin{proposition}[{\cite[Lemma 2.5]{Roberts_2013_Stability_of_Relative_Eq}}]
    Provided that 
    \begin{equation*}
        \sum_{i \neq j } \gamma_i \gamma_j \neq 0,
    \end{equation*}
    then $V \oplus V^\perp = \R^{2N}$, and both $V$ and $V^\perp$ are invariant spaces of $S$.
\end{proposition}
We can now formulate the second hypothesis.
\begin{hypothesis}\label{hyp:lyap}
We assume that $\sum_{i \neq j } \gamma_i \gamma_j \neq 0$ and that $Z^*$ is an equilibrium of the point-vortex dynamics in rotating frame~\eqref{eq:PVS_rot} which is linearly stable, in the sense that $S\big|_{V^\perp}$ can be block-diagonalized with blocks of the form $\begin{pmatrix} 0 & b \\ -b & 0\end{pmatrix}$.
\end{hypothesis}
More details about this definition are given in Section~\ref{sec:stab}.

\subsection{The confinement result}
With these two hypotheses, we prove the following long-time confinement result.
\begin{theorem}\label{theo:confinement}
    Let $\gamma_1,\ldots,\gamma_N \in \R\setminus\{0\}$, let $Z^* = (z_1^*,\ldots,z_N^*)$ satisfying both Hypotheses~\ref{hyp:strong_stabilité} and~\ref{hyp:lyap}. Then for every $(\omega_0^\eps)_{\eps >0}$ satisfying Hypothesis~\ref{hyp:omega} with $(z_1,\ldots,z_n) = Z^*$, for any $\beta<1/2$, for any $\alpha<\min(\beta/2,2-4\beta)$ and for any $\eps > 0$ small enough, the solution $\omega^\eps$ of the problem~\eqref{eq:du_pb} satisfies $\tau_{\eps,\beta} \ge \eps^{-\alpha}$.
\end{theorem}
Let us formulate a few remarks about Theorem~\ref{theo:confinement} and its hypotheses. As mentioned in the introduction, the big difference with the similar long-time confinement results obtained in \cite{marchiorobutta2018} and \cite{Donati_Iftimie_2020} is that we neither assume that $N=1$, nor that all the distances between the point-vortex grow quickly.

In \cite{Donati_Iftimie_2020}, the confinement result was obtained in a bounded domain $\Omega \subset \R^2$, for a single blob concentrated near a well-chosen point $x_0 \in \Omega$. This point must satisfy the existence of a conformal map $\mathcal{T} : \Omega \to B(0,1)$ such that $\mathcal{T}(x_0) = \mathcal{T}''(x_0) = \mathcal{T}'''(x_0) = 0$. In the present paper, Hypothesis~\ref{hyp:strong_stabilité} plays the role of this condition that ensures that the strain on the vortex blob is small. In \cite{marchiorobutta2018} it is the fact that point-vortices move away from each other that implies that the strain in each blob becomes small.

The main difference in the proof of Theorem~\ref{theo:confinement} compared to \cite{marchiorobutta2018} and \cite{Donati_Iftimie_2020} is that the control of the center of vorticity is much more delicate. Indeed, should for instance the point-vortex configuration be unstable, then the long-time confinement result would necessarily fail as the center of vorticity of the blobs would suffer from the same instability (see \cite{Donati_2024_DCDS} for an example of this situation). This is why we have to assume Hypothesis~\ref{hyp:lyap} which is the stability of the configuration. Hypothesis~\ref{hyp:lyap} is used in Section~\ref{sec:stab} to prove that the center of vorticity of each blob remains close to the initial configuration for a long enough time.

\subsection{Existence of the configuration}
Once Theorem~\ref{theo:confinement} is proved, we choose $\gamma_N$ so that the point-vortex configuration described in \eqref{eq:config} satisfies Hypotheses~\ref{hyp:strong_stabilité} and~\ref{hyp:lyap}.

For $N \ge 4$, let $\gamma_N$ be given by
\begin{equation}\label{def:gammaN}
    \gamma_N :=  \frac{(N-2)(N-6)}{12}.
\end{equation}
One can easily check in particular that such a $\gamma_N$ satisfies relations~\eqref{eq:range_gammaN}. We obtain the following.
 
\begin{theorem}\label{theo:construction}
    For every $N \ge 4$, the configuration described at relation~\eqref{eq:config} with $\gamma_N$ given by relation~\eqref{def:gammaN} satisfies both Hypotheses~\ref{hyp:strong_stabilité} and~\ref{hyp:lyap}. In the special case $N=6$ where $\gamma_6=0$, it must be understood in the sense that the configuration only consists of the $5$ exterior vortices.
\end{theorem}
    In conclusion, we construct one suitable configuration for each $N \in \{4, 7,8,\ldots\}$, and two configurations for $N=5$: the regular pentagon and the $4+1$ configuration which is a square of vortices with a central vortex, but no configuration for $N=6$. 
\begin{remark}
    Theorems~\ref{theo:confinement} and \ref{theo:construction} together prove Theorem~\ref{theo:main}.
\end{remark}

\section{Stability of relative equilibria}\label{sec:stab}

Before proving the full confinement result -- Theorem~\ref{theo:confinement} -- we start with the special case where the vorticity is itself a point-vortex. The problem becomes a question of stability for the dynamical system: if $|Z^\eps(0)-Z^*| \le \eps$, does there exist $\alpha$ such that $|Z^\eps(t) - Z^*(t)| \le \eps^\beta$ for every $t \le \eps^{-\alpha}$?

In this section the coordinates $Z$ are written in the frame rotating at angular velocity $\nu$, so the point-vortex dynamics is given by~\eqref{eq:PVS_rot} and $Z^*$ is stationary.

To then study the real confinement problem with desingularized vorticity, one must be able to state this stability result not only for point-vortices $Z^\eps$ but for the centers of vorticity of each blob (which will be denoted by $B^\eps=(b_1^\eps,\ldots,b_N^\eps$)), which only satisfy the point-vortex dynamics up to some error of order $\eps$. This is why in the following we introduce an error term in the point-vortex dynamics.

All the configurations $Z^*$ that we will construct in Section~\ref{sec:construction} are nonlinearly stable. However, the time interval on which we study the solution being small compared to the inverse of the size of the confinement, linear stability is enough to prove our result. This is why Hypothesis~\ref{hyp:lyap} only asks for linear stability.

This section is devoted to the proof of the following proposition, which will be used for the proof of the Theorem~\ref{theo:confinement} applied to the center of vorticity of each blob.

\begin{proposition}\label{prop:stabilité}
Let $Z^*$ satisfying Hypothesis~\ref{hyp:lyap}. For every $\eps \in (0,1)$, let $T_\eps >0$ and $Z^\eps : [0,T_\eps) \mapsto \R^{2N}$ be such that 
\begin{equation}\label{pb:zeps}
    \begin{cases}
        |Z^\eps(0)-Z^*| \le \eps \vspace{2mm}\\
        \displaystyle \left|\Gamma\der{}{t} Z^\eps(t) - J\nabla H(Z^\eps(t)) -\nu J\Gamma Z^\eps(t)\right| \le C \eps, \quad \forall t \ge 0.
    \end{cases}
\end{equation}
Let $\beta<1/2$ and
\begin{equation*}
        \bar{\tau}_{\eps,\beta} = \sup \left\{ t \in [0,T_\eps) \, , \, \forall s \in [0,t] \, , \, |Z^\eps(s)-Z^*| \le \eps^{\beta} \right\}.
\end{equation*}
Then for every $\alpha < \beta/2$, there exists $\eps_0$ such that for every $\eps \in (0,\eps_0)$,
\begin{equation*}
    \bar{\tau}_{\eps,\beta} \ge \min(T_\eps,\eps^{-\alpha})
\end{equation*}
and for every $t \le \min(T_\eps,\eps^{-\alpha})$,
\begin{equation*}
    |Z^\eps(s)-Z^*| \le C \eps^{2\beta-2\alpha}.
\end{equation*}
\end{proposition}
\begin{proof}
By Hypothesis~\ref{hyp:lyap}, for every $1\le i \le n-1$, there exists $W_i \subset V^\perp$ an invariant space of $V^\perp$ such that $V^\perp = \bigoplus_{i=1}^{n-1} W_i$ and $b_i$ such that
\begin{equation*}
    S_i:=S\big|_{W_i} = \begin{pmatrix} 0 & b_i \\ -b_i & 0\end{pmatrix}.
\end{equation*}
Let 
\begin{equation*}
    Z^\eps -Z^* = (y_1,y_2,Y_1,\ldots,Y_{N-1}) \in \vect\{Z^*\}\times \vect\{JZ^*\} \times W_1\times\ldots\times W_{n-1}. 
\end{equation*}
For every $i \in \{1,\ldots,N-1\}$ and $t \in [0,T_\eps)$, we have that
\begin{align*}
    \Big|\der{}{t} |Y_i(t)|^2\Big| & = 2 \Big|Y_i(t) \cdot \der{}{t} Y_i(t)\Big| \\
    & \le 2 \Big| Y_i(t) \cdot (J\nabla H(Z^\eps(t)) +\nu J\Gamma)Y_i(t)\Big| + 2 \Big|Y_i(t) \cdot \big( \der{}{t} Y_i(t) - (J\nabla H(Z^\eps(t)) +\nu J\Gamma)Y_i(t)\big) \Big| \\
    & \le  2 \Big| Y_i(t) \cdot S_iY_i(t) + Y_i(t)\cdot \mathcal{O}(|Z^\eps(t)-Z^*|^2) \Big| + C \eps|Y_i(t)| \\
    & \le |Y_i(t)| \big( \mathcal{O}(|Z^\eps(t)-Z^*|^2)+\eps\big),
\end{align*}
where we used the fact that $ Y_i \cdot S_iY_i  = 0$.
Since $|Z^\eps(t)-Z^*| \le \eps^\beta$ for every $t \le \bar{\tau}_{\eps,\beta}$ and since $\beta < 1/2$, we have that
\begin{equation*}
    \der{}{t} Y_i(t) = \mathcal{O}(\eps^{2\beta}).
\end{equation*}
In conclusion, for every $\alpha>0$ and every $t \le \min(\bar{\tau}_{\eps,\beta} , \eps^{-\alpha})$, we obtain that
\begin{equation*}
    |Y_i(t)| \le C \eps^{2\beta-\alpha}.
\end{equation*}
Now we recall from~\cite{Roberts_2013_Stability_of_Relative_Eq} that
\begin{equation*}
    S\big|_V = \begin{pmatrix} 0 & 0 \\ 2 \nu & 0 \end{pmatrix}.
\end{equation*}
We have the following system:
\begin{equation*}
\begin{split}
    \der{}{t} y_1 & = \mathcal{O}(|Z^\eps(t)-Z^*|^2) + \mathcal{O}(\eps)\\
    \der{}{t} y_2 & = 2\nu y_1 +  \mathcal{O}(|Z^\eps(t)-Z^*|^2) + \mathcal{O}(\eps),
\end{split}
\end{equation*}
from which we deduce that $|y_1(t)| \le C\eps^{2\beta-\alpha}$ and thus that $|y_2(t)| \le C \eps^{2\beta-2\alpha}$.
In conclusion, $|Z^\eps(t)-Z^*| \le C \eps^{2\beta-2\alpha}$ for every $t \le \min(\bar{\tau}_{\eps,\beta},\eps^{-\alpha})$ and thus by definition of $\bar{\tau}_{\eps,\beta}$ for every $\alpha < \beta/2$ and $\eps$ small enough,
\begin{equation*}
    \bar{\tau}_{\eps,\beta} \ge \eps^{-\alpha}.
\end{equation*}
\end{proof}

\section{Confinement}\label{sec:confinement}
In this section we prove Theorem~\ref{theo:confinement}. Some parts of the proof can be directly adapted from \cite{Donati_Iftimie_2020}, so we recall the main steps and only detail the parts that differ from \cite{Donati_Iftimie_2020}.

\subsection{Plan of the proof}

Let $\gamma_1,\ldots,\gamma_N \in \R\setminus\{0\}$ and $Z^* \in (\R^2)^N$ satisfying Hypotheses~\ref{hyp:strong_stabilité} and~\ref{hyp:lyap}, and let $(\omega_0^\eps)_{\eps>0}$ satisfying Hypothesis~\ref{hyp:omega} for this configuration $Z^*$. For every $\eps > 0$, let $t \mapsto \omega^\eps(t)$ be the $L^1 \cap L^\infty$ solution of \eqref{eq:du_pb}. Let $\alpha < \min(\beta/2, 2-4\beta)$. Let $\eps > 0$ which will be chosen small enough. We work on the time interval $[0, \min(\tau_{\eps,\beta},\eps^{-\alpha})]$ and prove that necessarily, $\min(\tau_{\eps,\beta},\eps^{-\alpha}) = \eps^{-\alpha}$, namely that $\tau_{\eps,\beta} \ge \eps^{-\alpha}$.

For each $i \in \{1,\ldots,N\}$, we start by introducing the exterior field
\begin{equation}\label{def:F}
    F_i^\eps(x,t) = \int_{\R^2} \sum_{j\neq i} K(x-z) \omega_j^\eps(z,t)\dd z
\end{equation}
so that each blob $\omega_i^\eps$ satisfies
\begin{equation}\label{eq:pb_u_i}
    \begin{cases}
        \displaystyle \partial_t \omega_i^\eps + (u_i^\eps + F_i^\eps) \cdot \nabla  \omega_i^\eps = 0, \\
        u_i^\eps(t) =  K \star \omega_i^\eps(t).
    \end{cases}
\end{equation}
We then introduce the centers of vorticity of each blob
\begin{equation*}
    b_i^\eps(t) = \frac{1}{\gamma_i}\int_{\R^2}  x \omega_i^\eps(x,t)\dd x,
\end{equation*}
the normalized second moment
\begin{equation*}
    I_i^\eps(t) = \frac{1}{\gamma_i}\int_{\R^2}  |x-b_i^\eps(t)|^2 \omega_i^\eps(x,t)\dd x,
\end{equation*}
which is positive by Hypothesis~\ref{hyp:omega}.
Finally we denote by $$B^\eps = (b_1^\eps,\ldots,b_N^\eps)$$ and $$I^\eps = \max_{1\le i \le N} I_i^\eps.$$

The first step we have to perform is to prove the following Lipschitz estimate for the exterior fields $F_i^\eps$.

\begin{lemma}\label{lem:lip_amélioré}
    There exists a constant $C$ such that for every $i \in \{ 1, \ldots, N\}$, for every $t \le \min(\tau_{\eps,\beta},\eps^{-\alpha})$ and for every $x,y \in B(z_i^*(t),\eps^\beta)$,
    \begin{equation*}
        \big|F_i^\eps(x,t) - F_i^\eps(y,t) \big| \le C \eps^\beta |x-y|.
    \end{equation*}
\end{lemma}
We prove it in Section~\ref{sec:lip}. From this lemma, we deduce estimates on $B^\eps$ and $I^\eps$ as follows.
\begin{lemma}\label{lem:moments}
    There exists a constant $C$ such that for every $t \le \min(\tau_{\eps,\beta},\eps^{-\alpha})$, $$I^\eps(t) \le C \eps^2$$ and $$|B^\eps(t)-Z^*(t)| \le C\eps^{2\beta-2\alpha}.$$
\end{lemma}
We prove it in Section~\ref{sec:proof_control}. Although the last estimate varies from what was obtained in \cite{Donati_Iftimie_2020}, namely that $|B^\eps(t)-Z^*(t)| \le C \eps$ (with $N=1$), what matters is only that $|B^\eps(t)-Z^*(t)| = o(\eps^\beta)$.

Let us define
\begin{equation*}
    R_i^\eps(t) = \max\big\{|x-b_i^\eps(t)|;\ x\in \supp \omega_i^\eps(t) \big\}
\end{equation*}
and choose some $X(t)\in  \supp \omega_i(t)$ such that $|X(t)-B_i(t)|=R_i(t)$. We denote by $s\mapsto X_t(s)$ the trajectory passing through $X(t)$ at time $t$ so that $X_t(t)=X(t)$. Then, we have the following lemma used to estimate the growth of the support.
\begin{lemma}\label{lem:vitesse_radiale}
For any $t\leq 	\tau_{\eps,\beta}$ we have that
\begin{equation}\label{eq:support}
    \der{}{s}|X_t(s)-b_i^\eps(s)|\Big|_{s=t} \le C_1 \eps^\beta R_i^\eps(t) + C_2\frac{I(t)}{(R_i^\eps(t))^3} + C_3\left(M\eps^{-\eta} \int_{|x-b_i^\eps|>R_i^\eps(t)/2}\omega_i^\eps(x,t)\dd x\right)^\frac{1}{2}.
\end{equation}
\end{lemma}
Equipped with Lemma~\ref{lem:lip_amélioré}, we skip the proof of Lemma~\ref{lem:vitesse_radiale} since it is completely identical to \cite[Lemma 4.2]{Donati_Iftimie_2020}. The final term we need to estimate is the integral in relation~\eqref{eq:support}.
\begin{lemma}\label{lem:4n}
For every $k \ge 1$ and $t\leq \min(\tau_{\eps,\beta},\eps^{-\beta})$ there exists a small constant $\eps_0=\eps_0(k)$, a large constant $C(k)$ and a constant $K_4$ which depends only $Z^*$, such that if 
\begin{equation*}
\eps\leq\eps_0\qquad\text{and}\qquad    r^4 \ge K_4\eps^{2}\left( 1 + kt\ln(2+t)\right),
\end{equation*}
then for every $i \in \{1,\ldots,N\}$,
\begin{equation*}
    \int_{|x-b_i^\eps|>r} \omega_i^\eps(x,t)\dd x \le  C(k)\frac{\eps^{k/2}}{r^k}.
\end{equation*}
\end{lemma}
We delay the proof to Section~\ref{sec:4n}.

Using Lemmas~\ref{lem:lip_amélioré}, \ref{lem:moments}, \ref{lem:vitesse_radiale} and \ref{lem:4n}, the end of the proof is nearly identical to \cite[Section 4.4]{Donati_Iftimie_2020}, with the only difference that we perform the argument on $$R^\eps(t) = \max_{1 \le i \le N} R_i^\eps(t)$$
to take into account that we have several blobs, which isn't the case in \cite{Donati_Iftimie_2020}. This concludes the proof of Theorem~\ref{theo:confinement} up to the proof of the intermediate lemmas.

\subsection{Preliminary tools}

As a consequence of the fact that $\omega_i^\eps$ is constant along the flow lines (which are well defined) of a divergence-free vector field $u_i^\eps+F_i^\eps$, see equation~\eqref{eq:pb_u_i}, we have the following derivation formula.
\begin{proposition}\label{prop:derivation}
    For any $\eps > 0$, for any $T>0$, for any $i \in \{1,\ldots,N\}$ and for any $\alpha \in C^1(\R^2\times[0,T],\R)$, we have that
    \begin{equation*}
        \der{}{t} \int_{\R^2} \alpha(x,t) \omega_i^\eps(x,t) = \int_{\R^2} \nabla \alpha(x,t) \cdot \big(u_i^\eps(x,t) + F_i^\eps(x,t)\big)\omega_i^\eps(x,t)\dd x + \int_{\R^2} \partial_t \alpha(x,t) \omega_i^\eps(x,t)\dd x.
    \end{equation*}
\end{proposition}
As a direct consequence, we have the two following relations for the time derivative of the moments of $\omega_i^\eps$:
\begin{equation*}
    \der{}{t} b_i^\eps(t) = \frac{1}{\gamma_i} \int  F_i^\eps(x,t)\omega_i^\eps(x,t)\dd x
\end{equation*}
and
\begin{equation}\label{eq:deriv_I}
    \der{}{t} I_i^\eps(t) = \frac{2}{\gamma_i} \int (x-b_i^\eps(t)) \cdot \Big(F_i^\eps(x,t)-F_i^\eps(b_i^\eps(t),t)\Big)\omega_j^\eps(x,t)\dd x.
\end{equation}
We also recall the following estimate, using the Cauchy-Schwarz inequality, that we will use without further reference in the paper:
\begin{equation*}
    \int |x-b_i^\eps(t)| \frac{\omega_i^\eps(x,t)}{\gamma_i} \dd x \le \left(\int  |x-b_i^\eps(t)|^2 \frac{\omega_i^\eps(x,t)}{\gamma_i} \dd x \int \frac{\omega_i^\eps(x,t)}{\gamma_i} \dd x \right)^{1/2} = \sqrt{I_i^\eps(t)}.
\end{equation*}

\subsection{Proof of Lemma~\ref{lem:lip_amélioré}}\label{sec:lip} Let $i \in\{1,\ldots,N\}$, $t \le \min(\tau_{\eps,\beta},\eps^{-\alpha})$ and $x,y \in D(z_i^*(t),\eps^\beta)$.
Recalling the definition \eqref{def:F} of the $F_i$, we have that
    \begin{align*}
    F_i^\eps(x,t) - F_i^\eps(y,t) & = \sum_{j\neq i} \int_{\R^2} \big(K(x-z)-K(y-z)\big) \omega_j^\eps(z,t) \dd z.
    \end{align*}
    From Hypotheses~\ref{hyp:omega} and the definition of $\tau_{\eps,\beta}$, we know that for $\eps$ small enough, for every $z \in \supp \omega_j^\eps(\cdot,t)$, $|x-z| \ge d/2$ and $|y-z| \ge d/2$, for $d = \min_{i \neq j} |z_i^*-z_j^*|$. Recall that $J_K$ denotes the Jacobian matrix of $K$. We expand using Taylor's formula by noticing that $K$ and its derivatives are bounded on $\R^2 \setminus D(0,d/2)$ to get that
    \begin{equation*}
        K(x-z)-K(y-z) = J_{K}(z_i^*(t)-z) (x-y) + \mathcal{O}(|x-y|^2)
    \end{equation*}
    and by the same argument that
    \begin{equation*}
        \big|J_{K}(z_i^*(t)-z) (x-y) - J_{K}(z_i^*(t)-z_j^*(t)) (x-y) \big| \le C |z-z_j^*||x-y|.
    \end{equation*}
    By Hypothesis~\ref{hyp:strong_stabilité}, we have that
    \begin{equation*}
        \sum_{j \neq i} J_{K}(z_i^*(t)-z_j^*(t)) (x-y) \int \omega_j^\eps(z,t) \dd z =  \sum_{j \neq i} \gamma_j  J_{K}(z_i^*(t)-z_j^*(t)) (x-y)  = 0.
    \end{equation*}
    Moreover, since $|x-y|\le 2 \eps^\beta$, since $t \le \tau_{\eps,\beta}$, and since for every $z \in \supp \omega_j^\eps$ we have that $|z-z_j^*(t)| \le \eps^\beta$, then
    \begin{equation*}
        \big|F_i^\eps(x,t) - F_i^\eps(y,t)\big| \le C \eps^\beta |x-y|.
    \end{equation*}

\subsection{Proof of Lemma~\ref{lem:moments}}\label{sec:proof_control}

    Equipped with Lemma~\ref{lem:lip_amélioré}, the proof of the estimate of $I_i^\eps $ is identical to the one of \cite[Lemma 4.3]{Donati_Iftimie_2020}, which we briefly recall. From relation~\eqref{eq:deriv_I} and Lemma~\ref{lem:lip_amélioré}, we have that
    \begin{align*}
    \left|\der{}{t} I_i^\eps(t)\right| & = 2\left|\int (x-b_i^\eps(t))  \cdot (F_i(x)-F_i(B_i(t)) \frac{\omega_i^\eps(x,t)}{\gamma_i} \dd x\right| \\
    & \le C \eps^\beta \int |x-b_i^\eps(t)|^2 \frac{\omega_i^\eps(x,t)}{\gamma_i} \dd x \\
    & \le C \eps^\beta I_i(t),
\end{align*}
from which we deduce by Gronwall's Lemma and Hypothesis~\ref{hyp:omega} that for $\eps$ small enough and for every $t \le \eps^{-\alpha} \le \eps^{-\beta}$, 
\begin{equation*}
    I(t) \le C I(0) \le C \eps^2.
\end{equation*}

We now estimate $|B^\eps - Z^*|$. We first compute that
\begin{align*}
    \der{}{t} b_i^\eps(t) - F_i^0(B^\eps(t)) & = \frac{1}{\gamma_i}\int F_i^\eps(x,t)\omega_i^\eps(x,t)\dd x - \sum_{j \neq i} \gamma_j K\big(b_i^\eps(t)-b_j^\eps(t)\big) \\
    & = \frac{1}{\gamma_i} \iint \sum_{j \neq i}\Big(K(x-y)-K\big(b_i^\eps(t)-b_j^\eps(t)\big)\Big)\omega_j^\eps(y,t)\omega_i^\eps(x,t)\dd y \dd x.
\end{align*}
Then, simply using that
\begin{equation*}
    \Big|K(x-y)-K\big(b_i^\eps(t)-b_j^\eps(t)\big)\Big| \le C |x-b_i^\eps(t)| + C|y-b_j^\eps(t)|
\end{equation*}
we conclude that
\begin{equation*}
    \left|\der{}{t} b_i^\eps(t) - F_i^0(B^\eps(t))\right| \le C \sqrt{I(t)} \le C \eps.
\end{equation*}
Writing this in a frame moving at velocity $\nu$, we conclude that $B^\eps$ satisfies the problem~\eqref{pb:zeps} on the time interval $[0,\min(\tau_{\eps,\beta},\eps^{-\alpha})]$. We thus use Proposition~\ref{prop:stabilité} to get that for every $t\le \min(\tau_{\eps,\beta},\eps^{-\alpha})$,
\begin{equation*}
    |B^\eps(t)-Z^*(t)|\le C \eps^{2\beta-2\alpha}.
\end{equation*}

\subsection{Proof of Lemma~\ref{lem:4n}}\label{sec:4n}

This Lemma is very reminiscent of \cite[Lemma 4.4]{Donati_Iftimie_2020} and the proof is very similar. We introduce the moments $m_{n,i}^\eps$ of order $4n$ of each blob $\omega_i^\eps$ given by
\begin{equation*}
    m_{n,i}^\eps(t) = \int |x-b_i^\eps(t)|^{4n} \omega_i^\eps(x,t)\dd x.
\end{equation*}
Computing the derivative in time using Proposition~\ref{prop:derivation} gives that
\begin{align*}
    \der{}{t} m_{n,i}^\eps(t) & =  \int \big( x - b_i^\eps(t) \big)\big|x-b_i^\eps(t)\big|^{4n-2} \cdot \big(u_i^\eps(x,t) + F_i^\eps(x,t) \big)\omega_i^\eps(x,t)\dd x \\
    & \hspace{1cm} -\int \der{}{t} b_i^\eps(t) \cdot \big( x - b_i^\eps(t) \big) \big|x-b_i^\eps(t)\big|^{4n-2}\omega_i^\eps(x,t)\dd x \\ 
    & = 4n \int \big( x - b_i^\eps(t) \big)\big|x-b_i^\eps(t)\big|^{4n-2} \cdot u_i^\eps(x,t) \omega_i^\eps(x,t)\dd x\\
    & \hspace{1cm} + 4n \int \big( x - b_i^\eps(t) \big)\big|x-b_i^\eps(t)\big|^{4n-2} \cdot F_i^\eps(x,t)\omega_i^\eps(x,t)\dd x \\
    & \hspace{2cm}  - 4n \int F_i^\eps(y,t) \omega_i^\eps(y,t) \dd y \cdot \int \big( x - b_i^\eps(t) \big) \big|x-b_i^\eps(t)\big|^{4n-2}\omega_i^\eps(x,t)  \dd x  \\
    & := a_{n,i}^\eps(t) + b_{n,i}^\eps(t)
\end{align*}
with
\begin{equation*}
    \begin{split}
        a_{n,i}^\eps(t) & = 4n \int \big( x - b_i^\eps(t) \big)\big|x-b_i^\eps(t)\big|^{4n-2} \cdot u_i^\eps(x,t) \omega_i^\eps(x,t)\dd x \\
        b_{n,i}^\eps(t) & = 4n \int \big( x - b_i^\eps(t) \big)\big|x-b_i^\eps(t)\big|^{4n-2} \cdot \left(F_i^\eps(x,t)-\int F_i^\eps(y,t) \omega_i^\eps(y,t) \dd y\right)\omega_i^\eps(x,t)\dd x.
    \end{split}
\end{equation*}
Replacing $u_i^\eps$ by its expression given in relations~\eqref{eq:pb_u_i}, we have that
\begin{align*}
    a_{n,i}^\eps(t) & = 4n\iint \big( x - b_i^\eps(t) \big)\big|x-b_i^\eps(t)\big|^{4n-2} \cdot K(x-y) \omega_i^\eps(x,t) \omega_i^\eps(y,t) \dd x \, \dd y.
\end{align*}
From \cite[pages 1726-1729]{Iftimie99Sideris} (see also \cite{Donati_Iftimie_2020}) we infer that the term $a_{n,i}^\eps(t)$ satisfies for every $i \in \{1,\ldots,N\}$ and for every $t \ge 0$ that
\begin{equation*}
    |a_{n,i}^\eps(t)| \le C n I_i^\eps(t) m_{n-1,i}^\eps(t).
\end{equation*}
Then, we compute that
\begin{align*}
    F_i^\eps(x,t)-\int F_i^\eps(y,t) \omega_i^\eps(y,t) \dd y & = \sum_{j \neq i}\iint \big[ K(x-z)-K(y-z)\big] \omega_j^\eps(z,t) \omega_i^\eps(y,t) \dd z \dd y \\
    & = \sum_{j \neq i}\iint \big[ J_K(x-z)(x-y) + \mathcal{O}(|x-y|^2)\big] \omega_j^\eps(z,t) \omega_i^\eps(y,t) \dd z \dd y \\ 
    & = \sum_{j \neq i} \int  \gamma_i J_K(x-z)(x-b_i^\eps(t)) \omega_j^\eps(z,t) \dd z  + \mathcal{O}(\eps^{2\beta}).
\end{align*}
Since
\begin{align*}
    J_K(x-z) & = J_K(z_i^*(t)-z_j^*(t)) + J_K(x-z_j^*(t)) - J_K(z_i^*(t)-z_j^*(t)) +  J_K(x-z) - J_K(x-z_j^*(t)) \\
    & = J_K(z_i^*(t)-z_j^*(t)) + \mathcal{O}(\eps^\beta),
\end{align*}
and since $|x-b_i^\eps(t)| = \mathcal{O}(\eps^\beta)$, we have that
\begin{align*}
    F_i^\eps(x,t)-\int F_i^\eps(y,t) \omega_i^\eps(y,t) \dd y & = \gamma_i \sum_{j \neq i} \int   J_K(z_i^*(t)-z_j^*(t))(x-b_i^\eps(t))\omega_j^\eps(z,t) \dd z + \mathcal{O}(\eps^{2\beta}) \\
    & = \gamma_i \sum_{j \neq i} \gamma_j J_K(z_i^*(t)-z_j^*(t))(x-b_i^\eps(t)) + \mathcal{O}(\eps^{2\beta}) \\
    & = \mathcal{O}(\eps^{2\beta}),
\end{align*}
where we used Hypothesis~\ref{hyp:strong_stabilité}. Therefore, we have that
\begin{equation*}
    |b_{n,i}^\eps(t)| \le Cn \eps^{5\beta} \int |x-b_i^\eps(t)|^{4n-4} \omega_i^\eps(x,t)\dd x = C n \eps^{5\beta} m_{n-1,i}^\eps(t).
\end{equation*}
Using Lemma~\ref{lem:moments} we obtain in the end that
\begin{equation*}
    \left|\der{}{t} m_{n,i}^\eps(t) \right| \le C n^2 (\eps^2 + \eps^{5\beta})m_{n-1,i}^\eps(t) \le C n^2 \eps^2 m_{n-1,i}^\eps(t)
\end{equation*}
assuming that $\beta > 2/5$, recalling that if Theorem~\ref{theo:confinement} is satisfied for some $\beta$, then it is satisfied for any $\beta' < \beta$. 

With this estimate in hand, the rest of the proof is identical to \cite[Lemma 4.4]{Donati_Iftimie_2020}.

\section{Choosing the intensity of the central vortex}\label{sec:construction}

This section is devoted to the proof of Theorem~\ref{theo:construction}.

Let us recall that we consider the following configuration of point-vortices: let $\gamma_1 = \ldots = \gamma_{N-1} = 1$ and $\gamma_N \in \R \setminus\{0\}$ to be chosen later, let $\zeta = e^{\ic \frac{2\pi}{N-1}}$ and let
\begin{equation*}
    \begin{cases}
        z_k^* = \zeta^k, \quad \forall k \in \{1,\ldots, N-1\},\\
        z_N^* = 0.
    \end{cases}
\end{equation*}
We recall that this configuration gives rise to a vortex crystal, see \cite{AREF_2003_Vortex_Crystals} for more details. For this configuration, recalling relation~\eqref{def:nu}, we compute that
\begin{equation*}
    \nu = \frac{(N-1)(N-2) + 2(N-1)\gamma_N}{2(N-1)} = \frac{N-2}{2}+ \gamma_N.
\end{equation*}

Our strategy is the following. We first prove the existence of each $N \ge 4$ of $\gamma_N$ such that the configuration satisfies Hypothesis~\ref{hyp:strong_stabilité}, then proves that $\gamma_N$ satisfies~\eqref{eq:range_gammaN} which by Theorem~\ref{theo:CS} implies the linear stability of the configuration in the sense of Hypothesis~\ref{hyp:lyap} -- which in addition requires to check that $\sum_{i\neq j}\gamma_i\gamma_j \neq 0$, which is equivalent to $\gamma_N \neq -\frac{N-2}{2}$.

\subsection{Reduction by symmetries}

Let
\begin{equation*}
    \mathcal{A}(i,t) = \sum_{j\neq i} \gamma_j J_{K}(z_i^*(t)-z_j^*(t)),
\end{equation*}
and let us recall that Hypothesis~\ref{hyp:strong_stabilité} is satisfied if and only if for every $t\ge 0$ and every $i\in \{1,\ldots,N\}$, we have that $\mathcal{A}(i,t)=0$. However, due to symmetries, we can reduce the computations to a single case.
\begin{lemma}\label{lem:reduc}
    Hypothesis~\ref{hyp:strong_stabilité} is satisfied if and only if
    \begin{equation*}
        \mathcal{A}(N-1,0) = 0.
    \end{equation*}
\end{lemma}
\begin{proof}
To avoid any confusion, we briefly go back to the matrix notation. The map $K$ satisfies for every $x \neq 0$ and every $\theta \in \R$ that
\begin{equation*}
    K(R_\theta x) = R_\theta K(x),
\end{equation*}
and differentiating this relation using the chain rule gives that
\begin{equation*}
    J_K(R_\theta x) R_\theta = R_\theta J_K(x).
\end{equation*}
Since for every $t\ge 0$, $z_i^*(t) = e^{\ic \nu t} z_i^*(0)$, this means that for every $i\in \{1,\ldots,N\}$ and every $t\ge0$,
\begin{equation*}
    \mathcal{A}(i,t) = R_{\nu t}\mathcal{A}(i,0) R_{-\nu t}
\end{equation*}
and thus,
\begin{equation*}
    \mathcal{A}(i,t) = 0 \Longleftrightarrow \mathcal{A}(i,0) = 0.
\end{equation*}
We now compute that for $i=N$, we have that
\begin{equation*}
    \mathcal{A}(N,0) = \sum_{j=1}^{N-1} J_K(-\zeta^j) = \sum_{j=0}^{N-2}  J_K(- R_{\frac{2\pi}{N-1}} \zeta^j) = R_{\frac{2\pi}{N-1}}\sum_{j=0}^{N-2} J_K(-\zeta^j) R_{-\frac{2\pi}{N-1}}=  R_{\frac{2\pi}{N-1}} \mathcal{A}(N,0)R_{-\frac{2\pi}{N-1}}.
\end{equation*}
Therefore, we always have that $\mathcal{A}(N,0) = 0$.
Finally, we have for any $i \in \{1,\ldots,N-1\}$ that
\begin{align*}
    \mathcal{A}(i,0) & = \sum_{\substack{j=1 \\ j \neq i}}^{N-1} J_K(\zeta^i-\zeta^j) + \gamma_N J_K(\zeta^i) \\
    & = R_{\frac{2\pi}{N-1}}\sum_{\substack{j=1 \\ j \neq i}}^{N-1} J_K(1-\zeta^{j-i}) R_{-\frac{2\pi}{N-1}} + \gamma_N R_{\frac{2\pi}{N-1}}J_K(1)R_{-\frac{2\pi}{N-1}}\\
    & = R_{\frac{2\pi}{N-1}}\left(\sum_{\substack{j=1}}^{N-2} J_K(1-\zeta^{j}) + \gamma_N J_K(1)\right)R_{-\frac{2\pi}{N-1}}\\
    & = R_{\frac{2\pi}{N-1}} \mathcal{A}(N-1,0)R_{-\frac{2\pi}{N-1}},
\end{align*}
which proves that for every $i\in\{1,\ldots,N-1\}$,
\begin{equation*}
    \mathcal{A}(i,0) = 0 \Longleftrightarrow \mathcal{A}(N-1,0) = 0
\end{equation*}
which concludes the proof.
\end{proof}

\subsection{Strain in vortex crystals}

We start with a general formula. We compute for $x$ close to $y \in \R^2\setminus\{0\}$,
\begin{align*}
    \frac{x^\perp}{|x|^{2}} - \frac{y^\perp}{|y|^{2}} & = \frac{1}{|y|^{2}}\left(x^\perp\frac{|y|^{2}}{|x-y+y|^{2}} - y^\perp \right) \\
    & = \frac{1}{|y|^{2}}\left(x^\perp\frac{|y|^{2}}{|x|^2 + 2 x\cdot y + |y|^2} - y^\perp \right) \\
    & = \frac{1}{|y|^{2}}\left(x^\perp\frac{1}{ \frac{|x-y|^2}{|y|^2} + 2 (x-y)\cdot \frac{y}{|y|^2} + 1} - y^\perp \right) \\
    & =  \frac{1}{|y|^{2}}\left(x^\perp\Big( 1-2 (x-y)\cdot \frac{y}{|y|^2}\Big) - y^\perp \right) + o(x-y) \\
    & =  \frac{1}{|y|^{2}}\left((x-y+y)^\perp\Big( 1-2 (x-y)\cdot \frac{y}{|y|^2}\Big) - y^\perp \right) + o(x-y) \\
    & =  \frac{1}{|y|^{2}}\left((x-y)^\perp - y^\perp2 (x-y)\cdot \frac{y}{|y|^2}\Big)\right) + o(x-y). \\
\end{align*}
Changing to the variable $x' = x-y$ without changing the name,
\begin{align*}
    \frac{(x+y)^\perp}{|x+y|^{2}} - \frac{y^\perp}{|y|^{2}} & = \frac{1}{|y|^{2}}\left(x^\perp - y^\perp 2 x\cdot \frac{y}{|y|^2}\Big)\right) + o(x).
\end{align*}
Therefore,
\begin{equation}\label{eq:formule_JK}
    J_K(y) = \frac{1}{2\pi|y|^4}\begin{pmatrix} 2y_1 y_2 &  y_2^2-y_1^2 \\ y_2^2-y_1^2   & -2y_1y_2  \end{pmatrix},
\end{equation}
and in particular, by denoting $\bar{y} = (y_1,-y_2)$ we have that
\begin{equation}\label{eq:formuleJK2}
    J_K(y) + J_K(\bar{y}) = \frac{(y_2^2-y_1^2)}{\pi|y|^4}\begin{pmatrix} 0 &  1 \\ 1   & 0  \end{pmatrix}.
\end{equation}
We are now ready to compute $\mathcal{A}(N-1,0)$. We group the terms by conjugates, except for $j = (N-1)/2$ when $N-1$ is even and obtain that
\begin{equation}\label{eq:JF1}
    \mathcal{A}(N-1,0) = \begin{cases} \displaystyle
        \gamma_N J_K(1) + \sum_{j=1}^{(N-2)/2} \big(J_K(1-\zeta^j) + J_K(1-\zeta^{N-j}) \big) & \text{ if $N$ is even}\\
        \displaystyle
        \gamma_N J_K(1) + \sum_{j=1}^{(N-3)/2} \big(J_K(1-\zeta^j) + J_K(1-\zeta^{N-j}) \big) + J_K(2) & \text{ if $N$ is odd.}
    \end{cases}
\end{equation}
From relation~\eqref{eq:formule_JK} we compute that
\begin{equation*}
    J_K(1) = -\frac{1}{2\pi}\begin{pmatrix} 0 &  1 \\ 1   & 0  \end{pmatrix}
\end{equation*}
and
\begin{equation*}
    J_K(2) = -\frac{1}{8\pi}\begin{pmatrix} 0 &  1 \\ 1   & 0  \end{pmatrix}
\end{equation*}
and thus using relation~\eqref{eq:formuleJK2}, we can write~\eqref{eq:JF1} under the more compact formulation
\begin{equation}\label{eq:JK3}
    \mathcal{A}(N-1,0) = \frac{1}{2\pi}\left(-\gamma_N + c_N - \frac{\sigma(N)}{4}\right)\begin{pmatrix} 0 &  1 \\ 1   & 0  \end{pmatrix}
\end{equation}
where 
\begin{equation*}
    \sigma(N) = \begin{cases}
        0 & \text{ if $N$ is even} \\
        1 & \text{ if $N$ is odd}
    \end{cases}
\end{equation*}
and
\begin{equation*}
    c_N = 2\sum_{j=1}^{\lfloor (N-2)/2\rfloor} \frac{\Im(\zeta^j)^2 - \big(1-\Re(\zeta^j)\big)^2}{|1-\zeta^j|^4}.
\end{equation*}
From relation~\eqref{eq:JK3} together with Lemma~\ref{lem:reduc}, we deduce that for every $N \ge 4$, there exists a unique $\gamma_N \in \R$, given by
\begin{equation*}
    \gamma_N = c_N - \frac{\sigma(N)}{4}, 
\end{equation*}
such that the configuration satisfies Hypothesis~\ref{hyp:strong_stabilité}. In the next subsection we will show that this last relation matches the definition of $\gamma_N$ given at relation~\eqref{def:gammaN}. Beware that at this step, it is possible that $\gamma_N = 0$ for some values of $N$ which would imply that the configurations have no central vortex.

\subsection{The value of $\gamma_N$}
Let $\gamma_N$ be given by the relation
\begin{equation*}
    \gamma_N = c_N - \frac{\sigma(N)}{4}.
\end{equation*}
It happens to exist a much simpler formula for $\gamma_N$. Let us first notice that
\begin{align*}
    c_N & = 2\sum_{j=1}^{\lfloor (N-2)/2\rfloor} \frac{\Im(\zeta^j)^2 - \big(1-\Re(\zeta^j)\big)^2}{|1-\zeta^j|^4} \\
    & = -2\sum_{j=1}^{\lfloor (N-2)/2\rfloor} \Re\left( \frac{1}{(1-\zeta^j)^2}\right) \\
    & = - \sum_{j \in \{1,\ldots,N-2\} \setminus \{(N-1)/2\}}\Re\left( \frac{1}{(1-\zeta^j)^2} \right).
\end{align*}
When $N$ is odd,
\begin{equation*}
    -\frac{\sigma(N)}{4} = -\frac{1}{4} = -\Re \left(\frac{1}{(1-\zeta^{(N-1)/2})^2}\right)
\end{equation*}
which in conclusion leads to
\begin{equation*}
    \gamma_N = - \Re\left(\sum_{j=1}^{N-2} \frac{1}{(1-\zeta^j)^2}\right).
\end{equation*}
Introducing the polynomial $p(z) = \frac{1-z^{N-1}}{1-z}=1+z+\ldots+z^{N-2}$, we have from one hand that
\begin{equation*}
    \sum_{j=1}^{N-2} \frac{1}{1-\zeta^j} = \frac{p'}{p} (1),
\end{equation*}
but more importantly in our situation, also that
\begin{equation*}
    \sum_{j=1}^{N-2} \frac{1}{(1-\zeta^j)^2} = -\left(\frac{p'}{p}\right)' (1).
\end{equation*}
We then compute that
\begin{align*}
    \left(\frac{p'}{p}\right)' (1) & = \frac{p''(z)}{p(z)} - \frac{(p'(z))^2}{p^2(z)} \\
    & = \frac{\frac{1}{6}(N-1)(N-2)(2N-3) - \frac{1}{2}(N-1)(N-2)}{N-1}-  \frac{\frac{1}{4} (N-1)^2(N-2)^2}{(N-1)^2} \\
    & = \frac{(N-2)(N-6)}{12}.
\end{align*}
In conclusion, we proved that
\begin{equation*}
    \gamma_N = -\Re\left(-\frac{(N-2)(N-6)}{12}\right) = \frac{(N-2)(N-6)}{12}.
\end{equation*}
From this expression, we see immediately that $\gamma_N = 0 \Longleftrightarrow N=6$, since we assumed that $N >2$, and that $\gamma_N > 0 \Longleftarrow N > 6$. Moreover, $\gamma_N \neq - \frac{N-2}{2}$. The first values of $\gamma_N$ are
\begin{equation*}
    \quad \gamma_4 = -\frac{1}{3}, \quad \gamma_5 = -\frac{1}{4}, \quad \gamma_6 = 0, \quad \gamma_7 = \frac{5}{12}, \quad \gamma_8 = 1,
\end{equation*}
The fact that $\gamma_6 =0$ means that for $N=6$, we produce a configuration that satisfies~\eqref{hyp:strong_stabilité}, but this configuration is a regular pentagon with no central vortex (it has only 5 vortices). Moreover, $\gamma_6$ satisfies relations~\eqref{eq:range_gammaN} and since in that case $\sum_{i\neq j} \gamma_i \gamma_j > 0$, this configuration without central vortex also satisfies Hypothesis~\ref{hyp:lyap}. So the conclusion of Theorem~\ref{theo:construction} remains true in the case $N=6$ by considering the constructed configuration as a $5$ point-vortices configuration.

\vspace{0.3cm}

\paragraph{Acknowledgments.}
\text{ }
\medskip

The author wishes to thank Thierry Gallay for numerous useful discussions, Michele Dolce for suggesting the method to compute the simplified expression of $\gamma_N$, and the anonymous referee for the very careful review, the comments and corrections.

This work was supported by the BOURGEONS project, grant ANR-23-CE40-0014-01 of the French National Research Agency (ANR).

\bibliographystyle{plain}

\end{document}